\documentclass[11pt]{article}
\usepackage{amsmath}
\usepackage{amsfonts}
\usepackage{graphicx}
\usepackage{amssymb}
\usepackage{cite}

\newtheorem{condition**}{A*}
\newtheorem{condition***}{C*}
\newtheorem{condition*}{C}

\newtheorem{proposition}{Proposition}[section]

\newtheorem{definition}{Definition}[section]
\newtheorem{theorem}{Theorem}[section]
\newtheorem{lemma}{Lemma}[section]
\newtheorem{remark}{Remark}[section]

\newenvironment{keywords}{{\bf Key words: }}{}

\begin{document}

\title{A Class of Linear-Quadratic-Gaussian (LQG) Mean-Field Game (MFG) of Stochastic Delay Systems}

\author{
Na Li \thanks{N. Li is with the Department of  Mathematics,
Qilu Normal University, Jinan. N. Li acknowledges the financial support partly by the Project B-Q34X and G-YL04. }
\quad \quad Shujun Wang \thanks{S. Wang is with the Department of Applied Mathematics,
The Hong Kong Polytechnic University, Hong Kong. }}

\maketitle

\begin{abstract}
This paper investigates the linear-quadratic-Gaussian (LQG) mean-field game (MFG) for a class of stochastic delay systems. We consider a large population system in which the dynamics of each player satisfies some forward stochastic differential delay equation (SDDE). The consistency condition or Nash certainty equivalence (NCE) principle is established through an auxiliary mean-field  system of anticipated forward-backward stochastic differential equation with delay (AFBSDDE). The wellposedness of such consistency condition system can be further established by some continuation method instead the  classical fixed-point analysis. Thus, the consistency condition maybe given on arbitrary time horizon. The decentralized strategies are derived which are shown to satisfy the $\epsilon$-Nash equilibrium property. Two special cases of our MFG for delayed system are further investigated.
\end{abstract}


\section{Introduction}

Recently, within the context of noncooperative game theory, the dynamic optimization of stochastic large-population system has attracted consistent and intense research attentions through a variety of fields including management science, engineering, mathematical finance and economics, social science, etc. The most special feature of controlled large-population system lies in the existence of considerable insignificant agents whose dynamics and (or) cost functionals are coupled via the state-average across the whole population. To design low-complexity strategies, one efficient methodology is the mean-field game (MFG) theory which enable us to obtain the decentralized control based on the individual own state together with some off-line quantity. The interested readers may refer \cite{GLL10,LL07} for the motivation and methodology, and \cite{AD,BCQ,BDL,CD} for recent progress in mean-field game theory. Besides, some other recent literature include \cite{B12,B11,hcm07,HCM12,hmc06,LZ08} for linear-quadratic-Gaussian (LQG) mean-field games of large-population system.

It is remarkable that all agents in above literature are comparably negligible in that they are not able to affect the whole population in separable manner. By contrast, their impacts are imposed in an unified manner through the population state-average. In this sense, all agents can be viewed as negligible peers but they can generate some mass effects via some ``unified manner" such as the control (input)-average or state (output)-average. These averages represent some type of impact imposed to other peers.

We point out in above works, all agents' states are formulated by (forward) stochastic differential equations (SDEs) with the initial conditions as a priori. As a sequel, the agents' objectives are minimizations of cost functionals involving their terminal states. In some realistic situation, there
exist some phenomena in which the state behavior depends not only on the
situation at time $t$, but also on a finite lagged state at $t-\theta.$
Moreover, if we use the information which we know to anticipate the
future evolution, we can get better results. As the novelty, this paper turns to consider the delay framework in which the agents' dynamics is characterized by some (forward) stochastic differential equations with delay (SDDEs). It means that the impacts are hardly imposed to each agent immediately. A new type of BSDEs called anticipated BSDEs (ABSDEs) was introduced in \cite{Peng-Yang}, which type of BSDEs can be applied to many fields such as optimal control and finance. Based on it, the problems which depend not only the present
but also the history were solved by \cite{Chen-Wu}. In the consequent works, the FBSDEs with delay and related LQ problems were studied in \cite{Chen-Wu2} and \cite{Chen-Wu-Yu}. A kind of stochastic maximum principle for optimal control problems of delay systems
involving continuous and impulse controls was considered in \cite{Yu}. The forward-backward linear quadratic stochastic optimal control problem
with delay was investigated in \cite{Huang-Li-Shi}. And the maximum principle for optimal control of fully coupled forward stochastic differential delayed equations was derived in \cite{Huang-Shi}. Moreover, some other important phenomena with delay were under consideration in \cite{Zhang1, Zhang2}.

To formulate the above problem mathematically, some SDDE should be introduced to characterize the dynamics of the agents. It is remarkable that there exist rich literature concerning the theories and applications of SDDE.
Generally, the large population problem with delay is under consideration. We discuss the related mean-field LQG games and derive the decentralized strategies. A stochastic process which relates to the delay term of control is introduced here to be the approximation of the control-average process. An auxiliary mean-field SDDE and a AFBSDDE system are considered and analyzed. Here, the AFBSDDE, which is composed by a SDDE and a ABSDE. Further, the AFBSDDE can be divided into two simple AFBSDDEs. In addition, the limit process is related to the wellposedness of a anticipated forward-backward ordinary differential equation with delay  (AFBODDE) and a AFBSDDE. We also derive the $\epsilon$-Nash equilibrium property of decentralized control strategy with $\epsilon=O(1/\sqrt N)$.

The rest of this paper is organized as follows. Section 2 formulates the large population LQG games of forward systems with delay. In Section 3, we derive the limiting optimal controls of the track systems and the consistency conditions. Section 4 is devoted to the related $\epsilon$-Nash equilibrium property. Section 5 gives two special cases in this work.

\section{Problem formulation}

$(\Omega,
\mathcal F, P)$ is a complete probability space on which a standard $(d+m\times N)$-dimensional Brownian motion $\{W^0_t,W^i_t,\ 1\le i\leq N\}_{0 \leq t \leq T}$ is defined, in which a finite time horizon $[0,T]$ is considered for fixed $T>0$. $\mathcal F^{W^0}_t:=\sigma\{W^0_s, 0\leq s\leq t\}$, $\mathcal F^{W^i}_t:=\sigma\{W^i_s, 0\leq s\leq t\}$, $\mathcal F^{i}_t:=\sigma\{W^0_s,W^i_s;0\leq s\leq t\}$. Here, $\{\mathcal F^{W^0}_t\}_{0\leq t\leq T}$ stands for the common information of all players; while $\{\mathcal F^{i}_t\}_{0\leq t\leq T}$ is  the individual information of $i^{th}$ player. Throughout this paper, $\mathbb{R}^n$
denotes the $n$-dimensional Euclidean space, its usual norm $|\cdot|$ and the usual inner product $\langle\cdot, \
\cdot\rangle$. For a given vector or matrix $M$, $M^\top$ stands for its transpose.

Moreover, we denote the spaces of matrices as follows.
\begin{itemize}
\item $S^d$ : the space of all $d\times d$ symmetric matrices.
\item $S^d_+$ : the subspace of all positive semi-definite  matrices of $S^d$.
\item $\hat{S}^d_+$ : the subspace of all positive definite  matrices of $S^d$.
\end{itemize}

For any Euclidean space $\mathbb R^n$, we introduce the following
notations:
\begin{itemize}
\item $L^2_{\mathcal F}(0,T;\mathbb R^n) = \{ g:[0,T]\times\Omega
\rightarrow \mathbb R^n\ |\ g(\cdot)$ is an $\mathbb R^n$-valued
$\mathcal{F}_t$-progressively measurable process such that $\|g\|^2_{L^2_{\mathcal F}}=\mathbb
E\int_0^T |g(t)|^2 dt <\infty\}$.
\item $L^2(0,T;\mathbb R^n) = \{ g:[0,T]
\rightarrow \mathbb R^n\ |\ g(\cdot)$ is an $\mathbb R^n$-valued
deterministic function such that $\|g\|^2_{L^2}=\int_0^{T}|g(t)|^{2}dt<\infty;$\}.
\item $L^\infty(0,T;\mathbb R^n) = \{ g:[0,T]\rightarrow \mathbb R^n\ |\ g(\cdot)$ is an $\mathbb R^n$-valued
uniformly bounded function\}.
\item $C(0,T;\mathbb R^n) = \{g:[0,T]\rightarrow \mathbb R^n\ |\ g(\cdot)$ is $\mathbb R^n$-valued continuous function\}.
\end{itemize}

In this paper, we consider a large population system with $N$ individual agents, denoted by $\{\mathcal{A}_{i}\}_{1 \leq i \leq N}$. The dynamics of $\mathcal{A}_{i}$ satisfies
the following controlled stochastic differential equation with delay (SDDE):
\begin{equation}\label{e1}
\left\{
\begin{aligned}
dx_t^i&=\Big[A_tx^i_t+\widetilde{A}_tx^i_{t-\delta}+B_tu^i_t+\widetilde{B}_tu^i_{t-\theta}+\frac{1}{N-1}
\sum_{j=1,j\neq i}^N\widehat B_tu^j_{t-\theta}\Big]dt+\sigma_tdW^i_t+\sigma^0_tdW^0_t,~t\in[0,T],\\
x^i_0&=a,~~~
x^i_t=\xi^i_t,~~~t\in[-\delta,0),~~~u^i_t=\eta^i_t,~~~t\in[-\theta,0),
\end{aligned}
\right.
\end{equation}
where
 $a$ is the initial state of $\mathcal{A}_i$, $x^i_{t-\delta}$ denotes the individual state delay, $u^i_{t-\theta}$ denotes the individual input or control delay. In addition, $\frac{1}{N-1}
\sum_{j=1,j\neq i}^N\widehat B_tu^j_{t-\theta}$ is introduced to denote the input delay of all other agents, imposed on a given agent $\mathcal{A}_i$. Similar state delay can be found in \cite{Sung}. Here, for simplicity, we assume all agents are statistically identical (homogeneous) in that they share the same coefficients $(A, \widetilde A, B, \widetilde{B},\widehat B, \sigma, \sigma^0)$ and deterministic initial state $a$. The admissible control strategy $u^i\in \mathcal{U}_i$, where$$
\mathcal{U}_i:=\Big\{u^i\big|u^i_t\in L^{2}_{\mathcal{F}^i_t}(0, T; \mathbb{R}^k)\Big\},\ 1\leq i \leq N.
$$
Let $u=(u^1, \cdots, u^{N})$ denotes the set of strategies of all $N$ agents; $u^{-i}=(u^1, \cdots, u^{i-1}$, $u^{i+1}, \cdots, u^{N})$ the strategies set but excluding that of $\mathcal{A}_i,1\leq i\leq N$. Considering the state and control delay, the cost functional for $\mathcal{A}_i,1\leq i\leq N$ is given by
\begin{equation}\label{e2}
\begin{aligned}
\mathcal{J}^i(u^i_t, u^{-i}_t)&=\frac{1}{2}\mathbb{E}\int_0^T\big[\langle
R_tx^i_t,x^i_t\rangle+\langle
\widetilde{R}_tx^i_{t-\delta},x^i_{t-\delta}\rangle+\langle N_tu^i_t,u^i_t\rangle+\langle \widetilde{N}_tu^i_{t-\theta},u^i_{t-\theta}\rangle\big]dt\\
&~~~+\frac{1}{2}\mathbb{E}\langle M
x^i_T,x^i_T\rangle,
\end{aligned}
\end{equation}
where $\widetilde R_t=0,~t\in[T,T+\delta],~\widetilde N_t=0,\ t\in[T,T+\theta]$.

For the coefficients of \eqref{e1} and \eqref{e2}, we set the following assumption:
\begin{description}
  \item[(H1)] $A_t,\widetilde{A}_t\in L^\infty(0,T;\mathbb{R}^{n\times n}),B_t,\widetilde{B}_t,\widehat{B}_t\in L^\infty(0,T;\mathbb{R}^{n\times k}),\sigma_t\in L^2(0,T;\mathbb{R}^{n\times m}),  \sigma^0_t\in L^2(0,T; \\\mathbb{R}^{n\times d}),a\in \mathbb {R}^n$;
  \item[(H2)] $R_t, \widetilde{R}_{t}\in L^\infty(0,T;S^n)$, $N_t, \widetilde{N}_{t}\in L^\infty(0,T; S^k)$, and $R(\cdot)+\widetilde{R}(\cdot+\delta)\in S_+^n$, for some $\delta>0$; $N(\cdot)+\widetilde{N}(\cdot+\theta)\in \hat S^n_+$ and the inverse $(N(\cdot)+\widetilde{N}(\cdot+\theta))^{-1}$ is also bounded  for some $\theta>0$; $M\in S_+^n$.
\end{description}

Now, we formulate the large population dynamic optimization problem with delay.\\

\textbf{Problem (LD).}
Find a control strategies set $\bar{u}=(\bar{u}^1,\cdots,\bar{u}^N)$ which satisfies
$$
\mathcal{J}^i(\bar{u}^i_t,\bar{u}^{-i}_t)=\inf_{u^i\in \mathcal{U}_i}\mathcal{J}^i(u^i_t,\bar{u}^{-i}_t),\ 0\leq i\leq N,
$$where $\bar{u}^{-i}$ represents $(\bar{u}^1,\cdots,\bar{u}^{i-1},\bar{u}^{i+1},\cdots, \bar{u}^N)$, for $1\leq i\leq N$.

\section{The limiting optimal control and Nash certainty equivalence (NCE) equation system}
To study Problem \textbf{(LD)}, an efficient approach is to discuss the associated mean-field games by analyzing the asymptotic behavior when the agent number $N$ tends to infinity. The key ingredient in this approach is to specify some suitable representation of state-average limit. With the help of such limit representation, we can figure out some auxiliary or tracking problem parameterized by the state-average limit. Based on it, the decentralized strategies of individual agents can thus be derived and we can also determine the state-average limit via some consistency condition. Moreover, the approximate Nash equilibrium property can be verified.

Noting that the agents are homogeneous, thus the optimal controls of $\mathcal{A}_i,1\leq i\leq N$ are conditionally independent with identical distribution. Suppose $\frac{1}{N-1}\sum_{j=1, j\neq i}^N\widehat B_tu^j_{t-\theta}$ is approximated by $m_0^{\theta}(t)\in\mathcal F_{t-\theta}^{W^0}$ as $N \rightarrow +\infty.$
Introducing the following auxiliary dynamics of the players,
\begin{equation}\label{e3}
\left\{
\begin{aligned}
d {x}_t^i&=\Big[A_t {x}^i_t+\widetilde{A}_t {x}^i_{t-\delta}
+B_tu^i_t+\widetilde{B}_tu^i_{t-\theta}+m_0^{\theta}(t)\Big]dt+\sigma_tdW^i_t+\sigma^0_tdW^0_t,~t\in[0,T],\\
 {x}^i_0&=a,~~~
 {x}^i_t=\xi^i_t,~~~t\in[-\delta,0),~~~u^i_t=\eta^i_t,~~~t\in[-\theta,0).
\end{aligned}
\right.
\end{equation} The associated limiting cost functional becomes
\begin{equation}\label{e4}
\begin{aligned}
J^i(u^i_t)&=\frac{1}{2}\mathbb{E}\int_0^T\big[\langle
R_t {x}^i_t, {x}^i_t\rangle+\langle
\widetilde{R}_t {x}^i_{t-\delta}, {x}^i_{t-\delta}\rangle+\langle N_tu^i_t,u^i_t\rangle+\langle \widetilde{N}_tu^i_{t-\theta},u^i_{t-\theta}\rangle\big]dt\\
&~~~+\frac{1}{2}\mathbb{E}\langle M
 {x}^i_T, {x}^i_T\rangle.
\end{aligned}
\end{equation}

Thus, we formulate the limiting LQG game with delay \textbf{(LLD)} as follows. \\

\textbf{Problem (LLD).} To find an
admissible control $\bar{u}^i\in \mathcal{U}_i$ for $i^{th}$ agent $\mathcal{A}_i$ satisfying
\begin{equation}\label{e5}
J^i(\bar{u}^i_t)=\inf_{u^i\in \mathcal{U}_i}J^i(u^i_t).
\end{equation}
Such an admissible control $\bar{u}^i$ is called an optimal control,
and $\bar x^i(\cdot)=x^i_{\bar{u}}(\cdot)$ is called the
corresponding optimal trajectory.

We link the Problem (LLD) to a stochastic Hamiltonian system as follows, which is an anticipated stochastic algebra differential
equation system with delay,
\begin{equation}\label{H sys}
\left\{
\begin{aligned}
0&=(N_t+\widetilde N_{t+\theta})\bar{u}^i_t+B_t^\top \bar{y}^i_t+\widetilde B^\top_{t+\theta}\mathbb E^{\mathcal F^i_t}[\bar{y}^i_{t+\theta}],\\
d\bar{x}_t^i&=\Big[A_t\bar{x}^i_t+\widetilde{A}_t\bar{x}^i_{t-\delta}-B_t\bar u^i_t-\widetilde{B}_t\bar u^i_{t-\theta}+m_0^{\theta}(t)\Big]dt+\sigma_tdW^i_t+\sigma^0_tdW^0_t,\\
d\bar{y}_t^i&=-\Big[A^\top_t\bar{y}^i_t+\widetilde{A}^\top_{t+\delta}\mathbb E^{\mathcal F^i_t}[\bar{y}^i_{t+\delta}]+(R_t+\widetilde{R}_{t+\delta})\bar{x}^i_t\Big]dt
+\bar{z}^i_tdW^i_t+\bar{z}^0_tdW^0_t,~t\in[0,T],\\
\bar{x}^i_0&=a,~~~
\bar{x}^i_t=\xi^i_t,~~~t\in[-\delta,0),~~~u^i_t=\eta^i_t,~~~t\in[-\theta,0),\\
\bar{y}^i_T&=M {x}^i_T,~~~
\bar{y}^i_t=0,~~~t\in(T, T+(\delta\vee\theta)].
\end{aligned}
\right.
\end{equation}

To get the optimal control of Problem \textbf{(LLD)}, we have the following theorem.

\begin{theorem}\label{l1}
Let \emph{\textbf{(H1)}}-\emph{\textbf{(H2)}} hold. The sufficient and necessary condition for the  optimal control of $\mathcal{A}_i$ for \emph{\textbf{(LLD)}} is that $u^i_t$ has the following form
\begin{equation}\label{e6}
   \bar{u}^i_t=-(N_t+\widetilde N_{t+\theta})^{-1}\big(B_t^\top \bar{y}^i_t+\widetilde B^\top_{t+\theta}\mathbb E^{\mathcal F^i_t}[\bar{y}^i_{t+\theta}]\big).
\end{equation}
Moreover, for any given $m_0^{\theta}(t)\in L^2_{\mathcal F_{t-\theta}^{W^0}}(-\theta,T+(\delta\vee\theta);\mathbb R^n)$, the stochastic Hamiltonian system \eqref{H sys} admits a unique solution
 $(\bar{x}_t^i, \bar u_t^i, \bar{y}_t^i,\bar{z}^i_t,\bar{z}^0_t)\in L^2_{\mathcal F^{i}_t}(-\delta,T;\mathbb R^n)\times \mathcal U_i\times L^2_{\mathcal F^{i}_t}(-\theta,T+(\delta\vee\theta);\mathbb R^n)\times L^2_{\mathcal F^{i}_t}(0, T;\mathbb R^{n\times m})\times L^2_{\mathcal F^{i}_t}(0,T;\mathbb R^{n\times d})$.
\end{theorem}

\begin{proof}
The sufficient and  necessary condition  part could be  from some variational calculus and dual representation, which is a straightforward consequence of the stochastic maximum principle in Yu \cite{Yu}. We omit the proof.

Moreover, under assumption \textbf{(H2)}, by the form of \eqref{e6}, our problem is to solve the following fully-coupled AFBSDDE,
\begin{equation}\label{e7}
\left\{
\begin{aligned}
d\bar{x}_t^i&=\Big[A_t\bar{x}^i_t+\widetilde{A}_t\bar{x}^i_{t-\delta}-B_t(N_t+\widetilde N_{t+\theta})^{-1}\big(B_t^\top \bar{y}^i_t+\widetilde B^\top_{t+\theta}\mathbb E^{\mathcal F^i_t}[\bar{y}^i_{t+\theta}]\big)\\
&~~~-\widetilde{B}_t(N_{t-\theta}+\widetilde N_{t})^{-1}\big(B_{t-\theta}^\top \bar{y}^i_{t-\theta}+\widetilde B_{t}\mathbb E^{\mathcal F^i_{t-\theta}}[\bar{y}^i_{t}]\big)+m_0^{\theta}(t)\Big]dt+\sigma_tdW^i_t+\sigma^0_tdW^0_t,\\
d\bar{y}_t^i&=-\Big[A^\top_t\bar{y}^i_t+\widetilde{A}^\top_{t+\delta}\mathbb E^{\mathcal F^i_t}[\bar{y}^i_{t+\delta}]+(R_t+\widetilde{R}_{t+\delta})\bar{x}^i_t\Big]dt
+\bar{z}^i_tdW^i_t+\bar{z}^0_tdW^0_t,~t\in[0,T],\\
\bar{x}^i_0&=a,~~~
\bar{x}^i_t=\xi^i_t,~~~t\in[-\delta,0),\\
\bar{y}^i_T&=M\bar{x}^i_T,~~~
\bar{y}^i_t=0,~~~t\in(T, T+(\delta\vee\theta)].
\end{aligned}
\right.
\end{equation}
Applying the classic ``continuation method" which  was proposed in \cite{HP}, \cite{PengWu99}, the proof is similar as in the Appendix of \cite{Chen-Wu-Yu}, the above linear AFBSDDE \eqref{e7} has a unique solution. So the Hamiltonian system (\ref{H sys}) admits a unique solution.
\end{proof}

For the further studying, consider the following two AFBSDDEs which are fully-coupled in states,
\begin{equation}\label{e8}
\left\{
\begin{aligned}
dx_t^{i,1}&=\Big[A_tx^{i,1}_t+\widetilde{A}_tx^{i,1}_{t-\delta}-B_t(N_t+\widetilde N_{t+\theta})^{-1}\big(B_t^\top y^{i,1}_t+\widetilde B^\top_{t+\theta}\mathbb E^{\mathcal F^{W^i}_t}[y^{i,1}_{t+\theta}]\big)\\
&~~~-\widetilde{B}_t(N_{t-\theta}+\widetilde N_{t})^{-1}\big(B_{t-\theta}^\top y^{i,1}_{t-\theta}+\widetilde B_{t}^\top\mathbb E^{\mathcal F^{W^i}_{t-\theta}}[y^{i,1}_{t}]\big)\Big]dt+\sigma_tdW^i_t,\\
dy_t^{i,1}&=-\Big[A^\top_ty^{i,1}_t+\widetilde{A}^\top_{t+\delta}\mathbb E^{\mathcal F^{W^i}_t}[y^{i,1}_{t+\delta}]+(R_t+\widetilde{R}_{t+\delta})x^{i,1}_t\Big]dt+z^i_tdW^i_t,~t\in[0,T],\\
x^{i,1}_0&=a^{i,1},~~~x^{i,1}_t=\xi^{i,1}_t,~~~t\in[-\delta,0),\\
y^{i,1}_T&=Mx^{i,1}_T,~~~y^{i,1}_t=0,~~~t\in(T, T+(\delta\vee\theta)],
\end{aligned}
\right.
\end{equation}
and
\begin{equation}\label{e9}
\left\{
\begin{aligned}
dx_t^2&=\Big[A_tx^2_t+\widetilde{A}_tx^2_{t-\delta}-B_t(N_t+\widetilde N_{t+\theta})^{-1}\big(B_t^\top y^2_t+\widetilde B^\top_{t+\theta}\mathbb E^{\mathcal F^{W^0}_t}[y^2_{t+\theta}]\big)\\
&~~~-\widetilde{B}_t(N_{t-\theta}+\widetilde N_{t})^{-1}\big(B_{t-\theta}^\top y^2_{t-\theta}+\widetilde B^\top_{t}\mathbb E^{\mathcal F^{W^0}_{t-\theta}}[y^2_{t}]\big)+m_0^{\theta}(t)\Big]dt+\sigma^0_tdW^0_t,\\
dy_t^2&=-\Big[A^\top_ty^2_t+\widetilde{A}^\top_{t+\delta}\mathbb E^{\mathcal F^{W^0}_t}[y^2_{t+\delta}]+(R_t+\widetilde{R}_{t+\delta})x^2_t\Big]dt+z^0_tdW^0_t,~t\in[0,T],\\
x^2_0&=a^2,~~~x^2_t=\xi^2_t,~~~t\in[-\delta,0),\\
y^2_T&=Mx^2_T,~~~y^2_t=0,~~~t\in(T, T+(\delta\vee\theta)],
\end{aligned}
\right.
\end{equation}
where $a^i=a^{i,1}+a^2$, $\xi_t^i=\xi_t^{i,1}+\xi_t^2$. It follows from the Appendix in \cite{Chen-Wu-Yu} that \eqref{e8} and \eqref{e9} admit the unique solutions $(x_t^{i,1},y_t^{i,1},z^i_t)\in L^2_{\mathcal F^{W^i}_t}(-\delta,T;\mathbb R^n)\times L^2_{\mathcal F^{W^i}_t}(-\theta,T+(\delta\vee\theta);\mathbb R^n)\times L^2_{\mathcal F^{W^i}_t}(0,T;\mathbb R^{n\times m})$ and $(x^2_t,y^2_t,z^0_t)\in L^2_{\mathcal F^{W^0}_t}(-\delta,T;\mathbb R^n)\times L^2_{\mathcal F^{W^0}_t}(-\theta,T+(\delta\vee\theta);\mathbb R^n)\times L^2_{\mathcal F^{W^0}_t}(0,T;\mathbb R^{n\times d})$. Then we have the following lemma.
\begin{lemma}\label{l2}
Let \emph{\textbf{(H1)-(H2)}} hold, if $(x_t^{i,1}, y_t^{i,1}, z^i_t)$ is the solution of \eqref{e8} and $(x_t^2, y_t^2, z_t^0)$ is the solution of \eqref{e9}, then   $(x_t^{i,1}+x_t^2, y_t^{i,1}+y_t^2, z^i_t, z_t^0)$ is the solution of \eqref{e7}.
\end{lemma}
\begin{proof}
It is easily to check that $\bar{x}_t^i=x_t^{i,1}+x_t^2$, $\bar{y}_t^i=y_t^{i,1}+y_t^2$, $\bar z_t^i=z_t^i$ and $\bar z^0_t=z^0_t$ are the solutions of AFBSDDE \eqref{e7}, then we can get the conclusion.
\end{proof}

In the following part, we will point out the essence of the limiting stochastic process $m_0^{\theta}(t)$.
Firstly, we introduce the following AFBODDE and AFBSDDE,
\begin{equation}\label{e12}
\left\{
\begin{aligned}
d[\mathbb E x_t^{1}]&=\Big[A_t[\mathbb Ex_t^{1}]+\widetilde{A}_t[\mathbb E x_{t-\delta}^{1}]-B_t(N_t+\widetilde N_{t+\theta})^{-1}\big(B_t^\top [\mathbb E y_t^{1}]+\widetilde B_{t+\theta}^\top[\mathbb E y_{t+\theta}^{1}]\big)\\
&~~~-\widetilde{B}_t(N_{t-\theta}+\widetilde N_{t})^{-1}\big(B_{t-\theta}^\top [\mathbb E y_{t-\theta}^{1}]+\widetilde B_{t}^\top[\mathbb E y^{1}_{t}]\big)\Big]dt,\\
d[\mathbb Ey_t^{1}]&=-\Big[A^\top_t[\mathbb Ey_t^{1}]+\widetilde{A}^\top_{t+\delta}[\mathbb Ey_{t+\delta}^{1}]+(R_t+\widetilde{R}_{t+\delta})[\mathbb Ex_t^{1}]\Big]dt,~t\in[0,T],\\
\mathbb E x^{1}_0&=a^{1},~~~~\mathbb Ex^{1}_t=\mathbb E\xi^{1}_t,~~~t\in[-\delta,0),\\
\mathbb Ey^{1}_T&=M[\mathbb Ex^{1}_T],~~~~\mathbb Ey^{1}_t=0,~~~t\in(T, T+(\delta\vee\theta)],
\end{aligned}
\right.
\end{equation}
and
\begin{equation}\label{e13}
\left\{
\begin{aligned}
dx_t^2&=\Big[A_tx^2_t+\widetilde{A}_tx^2_{t-\delta}-B_t(N_t+\widetilde N_{t+\theta})^{-1}\big(B_t^\top y^2_t+\widetilde B^\top_{t+\theta}\mathbb E^{\mathcal F^{W^0}_t}[y^2_{t+\theta}]\big)\\
&~~~-(\widetilde{B}_t+\widehat B_t)(N_{t-\theta}+\widetilde N_{t})^{-1}\big(B_{t-\theta}^\top y^2_{t-\theta}+\widetilde B_{t}^\top\mathbb E^{\mathcal F^{W^0}_{t-\theta}}[y^2_{t}]\big)\\
&~~~-\widehat B_t(N_{t-\theta}+\widetilde N_{t})^{-1}\big(B_{t-\theta}^\top[\mathbb E y^{1}_{t-\theta}]+\widetilde B_t^\top [\mathbb E y_t^{1}]\big)\Big]dt+\sigma^0_tdW^0_t,\\
dy_t^2&=-\Big[A^\top_ty^2_t+\widetilde{A}^\top_{t+\delta}\mathbb E^{\mathcal F^{W^0}_t}[y^2_{t+\delta}]+(R_t+\widetilde{R}_{t+\delta})x^2_t\Big]dt+z^0_tdW^0_t,~t\in[0,T],\\
x^2_0&=a^2,~~~x^2_t=\xi^2_t,~~~t\in[-\delta,0),\\
y^2_T&=Mx^2_T,~~~
y^2_t=0,~~~t\in(T, T+(\delta\vee\theta)].
\end{aligned}
\right.
\end{equation}

\begin{proposition}
$m_0^{\theta}(t)$ is in $L^2_{\mathcal F_t^{W^0}}(-\theta, T+(\delta\vee\theta); \mathbb R^n)$ and it is of the following form,
\[
\begin{aligned}
m_0^{\theta}(t)
&=-\widehat B_t(N_{t-\theta}+\widetilde N_{t})^{-1}\big[B_{t-\theta}^\top[\mathbb E y^{1}_{t-\theta}]+\widetilde B_t^\top[\mathbb E y_t^{1}]\big]\\
&\qquad-\widehat B_t(N_{t-\theta}+\widetilde N_{t})^{-1}\big[B_{t-\theta}^\top y^2_{t-\theta}+\widetilde B_t^\top\mathbb E^{\mathcal F_{t-\theta}^{W^0}}[y_t^2]\big],\\
\end{aligned}
\]
where $y_t^{1}$ is the solution of \eqref{e12} and $y_t^2$ is the solution of \eqref{e13}
\end{proposition}
\begin{proof}
 It follows from \eqref{e8} and \eqref{e9} that $y^{j,1}_{t}$ is independent of $W^0_t$, $y^{2}_{t}$ is independent of $W^j_t$, for $1\leq j\leq N$, respectively. Thus, we have
\begin{equation}\label{e10}
\mathbb E^{\mathcal F^j_{t-\theta}}[y^{j,1}_{t}]=\mathbb E^{\mathcal F^{W^j}_{t-\theta}}[y^{j,1}_{t}],\quad \mathbb E^{\mathcal F^j_{t-\theta}}[y^{2}_{t}]=\mathbb E^{\mathcal F^{W^0}_{t-\theta}}[y^{2}_{t}].\end{equation}
By virtue of Lemma \ref{l2}, we obtain
\begin{equation}\label{e11}
\begin{aligned}
m_0^{\theta}(t)&=\lim_{N\rightarrow\infty}\widehat B_t\frac{1}{N-1}\sum_{j=1,j\neq i}^{N}\bar{u}^j_{t-\theta}\\
&=-\widehat B_t\lim_{N\rightarrow\infty}\frac{1}{N-1}\sum_{j=1,j\neq i}^{N}(N_{t-\theta}+\widetilde N_{t})^{-1}(B_{t-\theta}^\top \bar {y}^j_{t-\theta}+\widetilde B_{t}^\top\mathbb E^{\mathcal F^j_{t-\theta}}[\bar{y}^j_{t}])\\
&=-\widehat B_t(N_{t-\theta}+\widetilde N_{t})^{-1}\lim_{N\rightarrow\infty}\frac{1}{N-1}\sum_{j=1,j\neq i}^{N}(B_{t-\theta}^\top (y^{j,1}_{t-\theta}+y^2_{t-\theta})+\widetilde B_{t}^\top\mathbb E^{\mathcal F^j_{t-\theta}}[(y^{j,1}_{t}+y^2_{t})])\\
&=-\widehat B_t(N_{t-\theta}+\widetilde N_{t})^{-1}\Big[\lim_{N\rightarrow\infty}\frac{1}{N-1}\sum_{j=1,j\neq i}^{N}(B_{t-\theta}^\top y^{j,1}_{t-\theta}+\widetilde B_{t}^\top\mathbb E^{\mathcal F^{W^j}_{t-\theta}}[y^{j,1}_{t}])\\
&\qquad+\lim_{N\rightarrow\infty}\frac{1}{N-1}\sum_{j=1,j\neq i}^{N}(B_{t-\theta}^\top y_{t-\theta}^2+\widetilde B_{t}^\top\mathbb E^{\mathcal F^{W^0}_{t-\theta}}[y^2_{t}])\Big]\\
&=-\widehat B_t(N_{t-\theta}+\widetilde N_{t})^{-1}\big[B_{t-\theta}^\top[\mathbb E y^{j,1}_{t-\theta}]+\widetilde B_t^\top[\mathbb E y_t^{j,1}]\big]\\
&\qquad-\widehat B_t(N_{t-\theta}+\widetilde N_{t})^{-1}\big[B_{t-\theta}^\top y^2_{t-\theta}+\widetilde B_t^\top\mathbb E^{\mathcal F_{t-\theta}^{W^0}}[y_t^2]\big]\\
&:=\Sigma_1^\theta(t)+\Sigma_2^\theta(t).
\end{aligned}
\end{equation}
Here, $\Sigma_1^\theta(t)=-\widehat B_t(N_{t-\theta}+\widetilde N_{t})^{-1}\big[B_{t-\theta}^\top[\mathbb E y^{j,1}_{t-\theta}]+\widetilde B_t^\top[\mathbb E y_t^{j,1}]\big]$, which is the deterministic function; $\Sigma_2^\theta(t)=-\widehat B_t(N_{t-\theta}+\widetilde N_{t})^{-1}\big[B_{t-\theta}^\top y^2_{t-\theta}+\widetilde B_t^\top\mathbb E^{\mathcal F_{t-\theta}^{W^0}}[y_t^2]\big]$ is in $L^2_{\mathcal F_t^{W^0}}(-\theta, T+(\delta\vee\theta); \mathbb R^n)$.
For $\mathbb E y^{i,1}_{t}=\mathbb E y^{j,1}_{t}$, for $i\neq j,\ i,j=1,2,\cdots,N$, so $\mathbb E y^{i,1}_{t}$ is independent on $i$, then we denote $\mathbb E y^{i,1}_{t}=\mathbb E y^{1}_{t}$,  where $y^{1}_{t}$ is the solution of \eqref{e12}.
 Thus $\Sigma_1^\theta(t)$ can be rewritten as $$\Sigma_1^\theta(t)=-\widehat B_t(N_{t-\theta}+\widetilde N_{t})^{-1}\big[B_{t-\theta}^\top[\mathbb E y^{1}_{t-\theta}]+\widetilde B_t^\top[\mathbb E y_t^{1}]\big].$$
Hence the result.
\end{proof}
\begin{remark}In what follows \eqref{e12}-\eqref{e13} are called the Nash certainty equivalence (NCE) equation system which can be used to determine the control state-average limit $m_0(t)$.
Note that $m_0(t)$ plays an important role due to the dependence of decentralized strategy $\bar{u}_i(t)$ on it. We can see that $\bar u_t^i$ in \eqref{e6} is dependent on the solution  $\bar y_t^i$ and $\bar y_{t+\theta}^i$ of \eqref{H sys}, and $\bar y_t^i$, $\bar y_{t+\theta}^i$ are dependent on $m_0(t)$.
\end{remark}

\section{$\epsilon$-Nash equilibrium analysis}
In above sections, we obtained the optimal control $\bar{u}^i_t, 1\le i\le N$ of Problem (\textbf{LLD}) through the consistency condition system. Now, we turn to verify the $\epsilon$-Nash equilibrium of Problem (\textbf{LD}). To start, we first present the definition of $\epsilon$-Nash equilibrium.
\begin{definition}\label{d1}
A set of controls $u_t^i\in \mathcal{U}_i,\ 1\leq i\leq N,$ for $N$ agents is called to satisfy an $\epsilon$-Nash equilibrium with respect to the costs $J^i,\ 1\leq i\leq N,$ if there exists $\epsilon\geq0$ such that for any fixed $1\leq i\leq N$, we have
\begin{equation}\label{e14}
J^i(\bar{u}_t^i,\bar{u}_t^{-i})\leq J^i(u_t^i,\bar{u}_t^{-i})+\epsilon,
\end{equation}
when any alternative control $u^i\in \mathcal{U}_i$ is applied by $\mathcal{A}_i$.
\end{definition}
\begin{remark}
If $\epsilon=0,$ then Definition \ref{d1} is reduced to the usual Nash equilibrium.
\end{remark}

Now, we state the main result of this paper and its proof will be given later.
\begin{theorem}\label{t2}
Under \emph{\textbf{(H1)-(H2)}}, $(\bar{u}_t^1,\bar{u}_t^2,\cdots,\bar{u}_t^N)$ satisfies the $\epsilon$-Nash equilibrium of \emph{(\textbf{LD})}. Here, for $1\le i\le N,$ $\bar{u}_t^i$ is given by
\eqref{e6}.
\end{theorem}
The proof of Theorem \ref{t2} needs several lemmas which are presented later. Denoting
 $\check {x}^i_t$ is the centralized state trajectory with respect to $\bar u^i_t$; $\hat{x}^i_t$ is
  the decentralized one with respect to $\bar u^i_t$. The cost functionals for \textbf{(LD)} and \textbf{(LLD)} are
   denoted by $\mathcal J^i(\bar u^i_t,\bar u^{-i}_t)$ and $J^i(\bar u^i_t)$,
   respectively. 
\begin{lemma}\label{l3}
\begin{flalign}
\sup_{1\leq i\leq N}\left[\sup_{0\leq t\leq T}\mathbb{E}\Big|\frac{1}{N-1}\sum_{j=1, j\neq i}^N\widehat B_t\bar u^j_{t-\theta}-m_0^{\theta}(t)\Big|^2\right]=O\Big(\frac{1}{N}\Big),\label{e16}
\end{flalign}
where $\bar u_t^j$ is given by \eqref{e6}.
\end{lemma}

{\it Proof.} By \eqref{e6}, \eqref{e10} and Lemma \ref{l2}, we get
\begin{equation}\label{e17}
\begin{aligned}
\frac{1}{N-1}\sum_{j=1, j\neq i}^N\widehat B_t\bar u^j_{t-\theta}&=\frac{1}{N-1}\sum_{j=1, j\neq i}^N\widehat B_t\Big\{-(N_{t-\theta}+\widetilde N_{t})^{-1}\big(B_{t-\theta}^\top \bar{y}^j_{t-\theta}+\widetilde B_{t}^\top\mathbb E^{\mathcal F^j_{t-\theta}}[\bar{y}^j_{t}]\big)\Big\}\\
&=-\widehat B_t(N_{t-\theta}+\widetilde N_{t})^{-1}\Bigg\{B_{t-\theta}^\top\Big(\frac{1}{N-1}\sum_{j=1, j\neq i}^Ny^{j,1}_{t-\theta}+y^2_{t-\theta}\Big)\\
&\qquad\qquad+\widetilde B_{t}^\top\Big(\frac{1}{N-1}\sum_{j=1, j\neq i}^N\mathbb E^{\mathcal F^{W^j}_{t-\theta}}[y^{j,1}_{t}]+\mathbb E^{\mathcal F^{W^0}_{t-\theta}}[y^{2}_{t}]\Big)\Bigg\}.
\end{aligned}
\end{equation}
Combining \eqref{e17} and \eqref{e11}, we obtain
\begin{equation}\label{e18}
\begin{aligned}
&\frac{1}{N-1}\sum_{j=1, j\neq i}^N\widehat B_t\bar u^j_{t-\theta}-m_0^{\theta}(t)\\
=&-\widehat B_t(N_{t-\theta}+\widetilde N_{t})^{-1}\Bigg\{B_{t-\theta}^\top\Big(\frac{1}{N-1}\sum_{j=1, j\neq i}^Ny^{j,1}_{t-\theta}-\mathbb E y^{1}_{t-\theta}\Big)\\
&\qquad\qquad\qquad\qquad\qquad+\widetilde B_{t}^\top\Big(\frac{1}{N-1}\sum_{j=1, j\neq i}^N\mathbb E^{\mathcal F^{W^j}_{t-\theta}}[y^{j,1}_{t}]-\mathbb E y_t^{1}\Big)\Bigg\}.
\end{aligned}
\end{equation}
Then it follows from \textbf{(H1)} that
\[
\begin{aligned}
&\mathbb E\Big|\frac{1}{N-1}\sum_{j=1, j\neq i}^N\widehat B_t\bar u^j_{t-\theta}-m_0^{\theta}(t)\Big|^2\\
\leq &C_1\Bigg\{\mathbb E\Big|\frac{1}{N-1}\sum_{j=1, j\neq i}^Ny^{j,1}_{t-\theta}-\mathbb E y^{1}_{t-\theta}\Big|^2+\mathbb E\Big|\frac{1}{N-1}\sum_{j=1, j\neq i}^N\mathbb E^{\mathcal F^{W^j}_{t-\theta}}[y^{j,1}_{t}]-\mathbb E y_t^{1}\Big|^2\Bigg\},
\end{aligned}
\]
where $C_1$ is a positive constant. Recall \eqref{e8} that $y^{j,1}_t\in  L^2_{\mathcal F^{W^j}_t}(-\theta,T+(\delta\vee\theta);\mathbb R^n)$. Thus $y^{j,1}_t$ is independent of $y^{k,1}_t$, for $j\neq k$, and we have $$\mathbb E \big(y^{j,1}_{t-\theta}-\mathbb E y^{1}_{t-\theta}\big)\big(y^{k,1}_{t-\theta}-\mathbb E y^{1}_{t-\theta}\big)=0$$
and $$\mathbb E \big(\mathbb E^{\mathcal F^{W^j}_{t-\theta}}[y^{j,1}_{t}]-\mathbb E y_t^{1}\big)\big(\mathbb E^{\mathcal F^{W^k}_{t-\theta}}[y^{k,1}_{t}]-\mathbb E y_t^{1}\big)=0.$$
Hence the result.    \hfill  $\Box$

\begin{lemma}\label{l4}
\begin{flalign}
&\sup_{1\leq i\leq N}\left[\sup_{0\leq t\leq T}\mathbb{E}\big|\check{x}^i_t-\hat{x}^i_t\big|^2\right]=O\Big(\frac{1}{N}\Big),\label{e19}\\
&\sup_{1\leq i\leq N}\left[\sup_{0\leq t\leq T}\mathbb{E}\big|\check{x}^i_{t-\delta}-\hat{x}^i_{t-\delta}\big|^2\right]=O\Big(\frac{1}{N}\Big).\label{e20}
\end{flalign}
\end{lemma}

{\it Proof.} For $\forall\ 1\leq i\leq N$, by \eqref{e1} and \eqref{e3}, we have
\begin{equation}\nonumber
\left\{
\begin{aligned}
d(\check{x}_t^i-\hat{x}_t^i)&=\Big[A_t(\check{x}_t^i-\hat{x}_t^i)+\widetilde{A}_t(\check{x}^i_{t-\delta}-\hat{x}^i_{t-\delta})+\frac{1}{N-1}
\sum_{j=1,j\neq i}^N\widehat B_t\bar{u}^j_{t-\theta}-m_0^{\theta}(t)\Big]dt,~t\in[0,T],\\
\check{x}^i_0-\hat{x}_0^i&=0,~~~
\check{x}^i_t-\hat{x}_t^i=0,~~~t\in[-\delta,0).
\end{aligned}
\right.
\end{equation}
Taking integral from 0 to $T$, we get
\begin{equation}\nonumber
\begin{aligned}
\check{x}_t^i-\hat{x}_t^i&=\int_0^T\Big[A_s(\check{x}_s^i-\hat{x}_s^i)+\widetilde{A}_s(\check{x}^i_{s-\delta}-\hat{x}^i_{s-\delta})+\frac{1}{N-1}
\sum_{j=1,j\neq i}^N\widehat B_s\bar{u}^j_{s-\theta}-m_0^\theta(t)\Big]ds.
\end{aligned}
\end{equation}
Note that $$\int_0^T\widetilde{A}_s(\check{x}^i_{s-\delta}-\hat{x}^i_{s-\delta})ds=\int_{-\delta}^{T-\delta}\widetilde{A}_{s+\delta}(\check{x}^i_{s}-\hat{x}^i_{s})ds
=\int_{0}^{T-\delta}\widetilde{A}_{s+\delta}(\check{x}^i_{s}-\hat{x}^i_{s})ds.$$
By Lemma \ref{l3}, \textbf{(H1)} and Gronwall's inequality, \eqref{e19} is obtained.

In addition, $$\sup_{0\leq t\leq T}\mathbb{E}\big|\check{x}^i_{t-\delta}-\hat{x}^i_{t-\delta}\big|^2=\sup_{0\leq \tau\leq T-\delta}\mathbb{E}\big|\check{x}^i_{\tau}-\hat{x}^i_{\tau}\big|^2\leq \sup_{0\leq \tau\leq T}\mathbb{E}\big|\check{x}^i_{\tau}-\hat{x}^i_{\tau}\big|^2.$$
Then we get \eqref{e20}.   \hfill  $\Box$

\begin{lemma}\label{l5}
\begin{flalign}
&\sup_{1\leq i\leq N}\left[\sup_{0\leq t\leq T}\mathbb{E}\Big||\check{x}^i_t|^2-|\hat{x}^i_t|^2\Big|\right]=O\Big(\frac{1}{\sqrt{N}}\Big),\label{e21}\\
&\sup_{1\leq i\leq N}\left[\sup_{0\leq t\leq T}\mathbb{E}\Big||\check{x}^i_{t-\delta}|^2-|\hat{x}^i_{t-\delta}|^2\Big|\right]=O\Big(\frac{1}{\sqrt{N}}\Big),\label{e22}\\
&\Big|\mathcal J^i(\bar u^i_t,\bar u^{-i}_t)-J^i(\bar u^i_t)\Big|=O\Big(\frac{1}{\sqrt{N}}\Big),\ \ 1\leq i\leq N. \label{e23}
\end{flalign}
\end{lemma}

{\it Proof.} For $\forall\ 1\leq i\leq N,$ it is easy to see $\sup\limits_{0\leq t\leq T}\mathbb{E}\big|\hat{x}^i_t\big|^2<+\infty,\sup\limits_{0\leq t\leq T}\mathbb{E}\big|\hat{x}^i_{t-\delta}\big|^2<+\infty$. Applying Cauchy-Schwarz inequality and \eqref{e19}, we have
\begin{equation}\nonumber\begin{aligned}
&\sup_{0\leq t\leq T}\mathbb{E}\left||\check{x}^i_t|^2-|\hat{x}^i_t|^2\right|\\
\leq&\sup_{0\leq t\leq T}\mathbb{E}\big|\check{x}^i_t-\hat{x}^i_t\big|^2+2\Big(\sup_{0\leq t\leq T}\mathbb{E}|\hat{x}^i_t|^2\Big)^{\frac{1}{2}}\Big(\sup_{0\leq t\leq T}\mathbb{E}\big|\check{x}^i_t-\hat{x}^i_t\big|^2\Big)^{\frac{1}{2}}\\
=&O\Big(\frac{1}{\sqrt{N}}\Big).
\end{aligned}
\end{equation}
Similarly, \eqref{e22} is obtained. Then noting \textbf{(H2)}, we have
\begin{equation}\nonumber\begin{aligned}
&\big|\mathcal J^i(\bar u^i_t,\bar u^{-i}_t)-J^i(\bar u^i_t)\big|\\
\leq& C_2\mathbb{E}\int_0^T \big(\big||\check{x}^i_t|^2-|\hat{x}^i_t|^2\big|+\big||\check{x}^i_{t-\delta}|^2-|\hat{x}^i_{t-\delta}|^2\big|\big)dt+ C_2\mathbb{E}\big||\check{x}^i_T|^2-|\hat{x}^i_T|^2\big|\\
=& O\Big(\frac{1}{\sqrt{N}}\Big),
\end{aligned}
\end{equation}
which implies \eqref{e23}. Here, $C_2$ is a positive constant. \hfill  $\Box$

Until now, we have addressed some estimates of states and costs corresponding to control $\bar{u}^i_t$, $1\le i\le N$. Next we will focus on the $\epsilon$-Nash equilibrium for (\textbf{LD}). For any fixed $i$, $1\le i\le N$, consider a alternative control $u^i_t \in \mathcal{U}_i$ for $\mathcal{A}_i$ and introduce the dynamics
\begin{equation}\label{e24}
\left\{
\begin{aligned}
dl_t^i&=\Big[A_tl^i_t+\widetilde{A}_tl^i_{t-\delta}+B_tu^i_t+\widetilde{B}_tu^i_{t-\theta}+\frac{1}{N-1}
\sum_{\kappa=1,\kappa\neq i}^N\widehat B_t\bar{u}^\kappa_{t-\theta}\Big]dt+\sigma_tdW^i_t+\sigma^0_tdW^0_t,~t\in[0,T],\\
l^i_0&=a,~~~
l^i_t=\xi^i_t,~~~t\in[-\delta,0),~~~u^i_t=\eta^i_t,~~~t\in[-\theta,0),
\end{aligned}
\right.
\end{equation}
whereas other players keep the control $\bar{u}_t^j,1\leq j\leq N,j\neq i$ i.e.,
\begin{equation}\nonumber
\left\{
\begin{aligned}
dl_t^j&=\Big[A_tl^j_t+\widetilde{A}_tl^j_{t-\delta}+B_t\bar{u}^j_t+\widetilde{B}_t\bar{u}^j_{t-\theta}+\frac{1}{N-1}
\widehat B_t\Big(\sum_{\kappa=1,\kappa\neq i,j}^N\bar{u}^\kappa_{t-\theta}+u^i_{t-\theta}\Big)\Big]dt+\sigma_tdW^j_t+\sigma^0_tdW^0_t,\\
l^j_0&=a,~~~
l^j_t=\xi^j_t,~~~t\in[-\delta,0),~~~u^i_t=\eta^i_t,~~~t\in[-\theta,0).
\end{aligned}
\right.
\end{equation}
The dynamics of $\mathcal{A}_i$ with respect to $u^i_t$ for (\textbf{LLD}) is
\begin{equation}\label{e25}
\left\{
\begin{aligned}
dp_t^i&=\Big[A_tp^i_t+\widetilde{A}_tp^i_{t-\delta}+B_tu^i_t+\widetilde{B}_tu^i_{t-\theta}+m_0^{\theta}(t)\Big]dt+\sigma_tdW^i_t+\sigma^0_tdW^0_t,~t\in[0,T],\\
p^i_0&=a,~~~
p^i_t=\xi^i_t,~~~t\in[-\delta,0),~~~u^i_t=\eta^i_t,~~~t\in[-\theta,0).
\end{aligned}
\right.
\end{equation}
We have the following lemma.
\begin{lemma}\label{l6}
\begin{flalign}
&\sup_{1\leq i\leq N}\left[\sup_{0\leq t\leq T}\mathbb{E}\big|l^i_t-p^i_t\big|^2\right]=O\Big(\frac{1}{N}\Big),\label{e26}\\
&\sup_{1\leq i\leq N}\left[\sup_{0\leq t\leq T}\mathbb{E}\big|l^i_{t-\delta}-p^i_{t-\delta}\big|^2\right]=O\Big(\frac{1}{N}\Big),\label{e27}\\
&\sup_{1\leq i\leq N}\left[\sup_{0\leq t\leq T}\mathbb{E}\Big||l^i_t|^2-|p^i_t|^2\Big|\right]=O\Big(\frac{1}{\sqrt{N}}\Big),\label{e28}\\
&\sup_{1\leq i\leq N}\left[\sup_{0\leq t\leq T}\mathbb{E}\Big||l^i_{t-\delta}|^2-|p^i_{t-\delta}|^2\Big|\right]=O\Big(\frac{1}{\sqrt{N}}\Big),\label{e29}\\
&\Big|\mathcal J^i(u^i_t,\bar u^{-i}_t)-J^i(u^i_t)\Big|=O\Big(\frac{1}{\sqrt{N}}\Big),\ \ 1\leq i\leq N. \label{e30}
\end{flalign}
\end{lemma}

{\it Proof.} Using the same analysis to the proof of Lemma \ref{l4}, by \eqref{e24}-\eqref{e25} and noting Lemma \ref{l3}, we get \eqref{e26} and \eqref{e27}. By virtue of \eqref{e26} and \eqref{e27}, \eqref{e28} and \eqref{e29} follows by applying Cauchy-Schwarz inequality. Same to Lemma \ref{l5}, \eqref{e30} is obtained.      \hfill  $\Box$

\textbf{Proof of Theorem \ref{t2}:} Now, we consider the $\epsilon$-Nash equilibrium for $\mathcal{A}_i,1\leq i\leq N$. It follows from \eqref{e23} and \eqref{e30} that
\begin{equation}\nonumber\begin{aligned}
\mathcal J^i(\bar u^i_t,\bar u^{-i}_t)&=J^i(\bar u^i_t)+O\Big(\frac{1}{\sqrt{N}}\Big)\\
&\leq J^i(u^i_t)+O\Big(\frac{1}{\sqrt{N}}\Big)\\
&=\mathcal J^i(u^i_t,\bar u^{-i}_t)+O\Big(\frac{1}{\sqrt{N}}\Big).
\end{aligned}
\end{equation}
Thus, Theorem \ref{t2} follows by taking $\epsilon=O\Big(\frac{1}{\sqrt{N}}\Big)$.

\section{Special cases}
In this section, we will study some special cases to show the essence of MFG problem with delay.\\

\textbf{ Case I:}
 In this case, we will give the ``closed-loop" form of the $\epsilon$-Nash equilibrium.
For simplicity, let $\widetilde A_t=\widetilde B_t=0$ in system \eqref{e1}, then we study the following system,

\begin{equation}\label{system x}
\left\{
\begin{aligned}
dx_t^i&=\Big[A_tx^i_t+B_tu^i_t+\frac{1}{N-1}
\sum_{j=1,j\neq i}^N\widehat B_tu^j_{t-\theta}\Big]dt+\sigma_tdW^i_t+\sigma^0_tdW^0_t,~t\in[0,T],\\
x_0&=a,\\
\end{aligned}
\right.
\end{equation}
and the cost functional is still  \eqref{e2}.

Now, we consider the following FBSDE
\begin{equation}\label{e77}
\left\{
\begin{aligned}
d\bar{x}_t^i&=\Big[A_t\bar{x}^i_t-B_t(N_t+\widetilde N_{t+\theta})^{-1}B_t^\top \bar{y}^i_t+m_0^{\theta}(t)\Big]dt+\sigma_tdW^i_t+\sigma^0_tdW^0_t,\\
d\bar{y}_t^i&=-\Big[A^\top_t\bar{y}^i_t+(R_t+\widetilde{R}_{t+\delta})\bar{x}^i_t\Big]dt
+\bar{z}^i_tdW^i_t+\bar{z}^0_tdW^0_t,~t\in[0,T],\\
\bar{x}^i_0&=a,\\
\bar{y}^i_T&=M\bar{x}^i_T,\\
\end{aligned}
\right.
\end{equation}

In system \eqref{e77}, we could deduce the $m_0^{\theta}(t)$ as follows,
\[
\begin{aligned}
m_0^{\theta}(t)&=\lim_{N\rightarrow\infty}\widehat B_t\frac{1}{N-1}\sum_{j=1,j\neq i}^{N}\bar{u}^j_{t-\theta}\\
&=-\widehat B_t(N_{t-\theta}+\widetilde N_{t})^{-1}B_{t-\theta}^{\top}\lim_{N\rightarrow\infty}\frac{1}{N-1}\sum_{j=1,j\neq i}^{N}\tilde y_{t-\theta}^i\\
&=-\widehat B_t(N_{t-\theta}+\widetilde N_{t})^{-1}B_{t-\theta}^{\top}\mathbb E^{\mathcal F^{W^0}_{t-\theta}}[\tilde y_{t-\theta}]
\end{aligned}
\]
where $\tilde y_{t}$ satisfies the following FBSDDE,
\begin{equation}\label{e777}
\left\{
\begin{aligned}
d\tilde{x}_t&=\Big[A_t\tilde{x}_t-B_t(N_t+\widetilde N_{t+\theta})^{-1}B_t^\top \tilde{y}_t-\widehat B_t(N_{t-\theta}+\widetilde N_{t})^{-1}B_{t-\theta}^{\top}\mathbb E^{\mathcal F^{W^0}_{t-\theta}}[\tilde y_{t-\theta}]\Big]dt+\sigma_tdW_t\\
&\qquad\qquad+\sigma^0_tdW^0_t,\\
d\tilde{y}_t&=-\Big[A^\top_t\tilde{y}_t+(R_t+\widetilde{R}_{t+\delta})\tilde{x}_t\Big]dt
+\tilde{z}_tdW_t+\tilde{z}^0_tdW^0_t,~t\in[0,T],\\
\tilde{x}_0&=a,\\
\tilde{y}_T&=M\tilde{x}_T,\\
\end{aligned}
\right.
\end{equation}
where $W_t$ and $W^i_t$ are independent and identically distributed. According to the Appendix in \cite{Chen-Wu-Yu}, \eqref{e777} has a unique solution $(\tilde x_t, \tilde y_t, \tilde z_t, \tilde z_t^0)$.

Then, FBSDDE \eqref{e77} could be rewritten as follows
\begin{equation}\label{e7777}
\left\{
\begin{aligned}
d\bar{x}_t^i&=\Big[A_t\bar{x}^i_t-B_t(N_t+\widetilde N_{t+\theta})^{-1}B_t^\top \bar{y}^i_t-\widehat B_t(N_{t-\theta}+\widetilde N_{t})^{-1}B_{t-\theta}^{\top}\mathbb E^{\mathcal F^{W^0}_{t-\theta}}[\tilde y_{t-\theta}]\Big]dt]dt\\
&\qquad\qquad+\sigma_tdW^i_t+\sigma^0_tdW^0_t,\\
d\bar{y}_t^i&=-\Big[A^\top_t\bar{y}^i_t+(R_t+\widetilde{R}_{t+\delta})\bar{x}^i_t\Big]dt
+\bar{z}^i_tdW^i_t+\bar{z}^0_tdW^0_t,~t\in[0,T],\\
\bar{x}^i_0&=a,\\
\bar{y}^i_T&=M\bar{x}^i_T.\\
\end{aligned}
\right.
\end{equation}
FBSDE \eqref{e7777} could be decoupled by the following  Riccati equation and ordinary differential equation
\[
\left\{
\begin{aligned}
&~\dot{P}_t+P_tA_t+A^\top_tP_t+R_t+\widetilde{R}_{t+\delta}-P_tB_t(N_t+\widetilde{N}_{t+\theta})^{-1}B^\top_tP_t=0,\\
&~P_T=M,
\end{aligned}
\right.
\]
and
\[
\left\{
\begin{aligned}
&~\dot{\phi}_t+[A_t-B_t(N_t+\widetilde{N}_{t+\theta})^{-1}B^\top_tP_t]\phi_t-P_t\hat{B}_t(N_{t-\theta}+\widetilde{N}_{t})^{-1}B^\top_{t-\theta}
\mathbb E ^{\mathcal F_{t-\theta}^{W^0}}[\tilde y_{t-\theta}]=0,\\
&~\phi_T=0.
\end{aligned}
\right.
\]
We obtain the optimal feedback is
\[
\bar{u}_t^i=-(N_t+\widetilde{N}_{t+\theta})^{-1}B^\top_t(P_t\bar x_t^i+\phi_t).
\]
From Theorem \ref{t2}, we claim that $(\bar{u}_t^1,\bar{u}_t^2,\cdots,\bar{u}_t^N)$ is the  $\epsilon$-Nash equilibrium of the Problem \eqref{system x}, \eqref{e2}.\\

\textbf{Case II:}
Now, we consider another special case. Let $A_t= B_t=0$ in system \eqref{e1}, moreover, we assume $\delta=\theta$, then we study the following system
\begin{equation}\label{e222}
\left\{
\begin{aligned}
dx_t^i&=\Big[\widetilde{A}_tx^i_{t-\delta}
+\widetilde{B}_tu^i_{t-\delta}+\frac{1}{N-1}
\sum_{j=1,j\neq i}^N\widehat B_tu^j_{t-\delta}\Big]dt+\sigma_tdW^i_t+\sigma^0_tdW^0_t,~t\in[0,T],\\
x^i_0&=a,~~~
x^i_t=\xi^i_t,~~~t\in[-\delta,0),~~~u^i_t=\eta^i_t,~~~t\in[-\theta,0),
\end{aligned}
\right.
\end{equation}
and the cost functional is
\[
\begin{aligned}
\mathcal{J}^i(u^i_t, u^{-i}_t)&=\frac{1}{2}\mathbb{E}\int_0^T\big[N_t(u^i_t)^2+ \widetilde{N}_t(u^i_{t-\theta})^2\big]dt+M
x^i_T.
\end{aligned}
\]
We will consider the following system instead of \eqref{e222}

\begin{equation}\label{e333}
\left\{
\begin{aligned}
d{x}_t^i&=\Big[\widetilde{A}_t{x}^i_{t-\delta}+\widetilde{B}_tu^i_{t-\delta}+m_0^{\delta}(t)
\Big]dt+\sigma_tdW^i_t+\sigma^0_tdW^0_t,~t\in[0,T],\\
{x}^i_0&=a,~~~
{x}^i_t=\xi^i_t,~~~t\in[-\delta,0),~~~u^i_t=\eta^i_t,~~~t\in[-\theta,0).
\end{aligned}
\right.
\end{equation}
and the adjoint equation is
\begin{equation}\label{adjiont}
\left\{
\begin{aligned}
d{y}_t^i&=-\widetilde{A}_{t+\delta}\mathbb E^{\mathcal F^i_t}[{y}^i_{t+\delta}]dt
+{z}^i_tdW^i_t+{z}^0_tdW^0_t,~t\in[0,T],\\
{y}^i_T&=-M,\\
{y}^i_t&=0,~~~t\in(T, T+(\delta\vee\theta)].
\end{aligned}
\right.
\end{equation}
We could solve \eqref{adjiont} explicitly by applying the method in \cite{Yu}, which can also be found in \cite{Menoukeu-Pamen}.\\

(i) When $t\in[T-\delta, T]$, the ABSDE \eqref{adjiont} becomes
\[
\bar y_t^i=-M-\int_t^T\bar z_s^idW_t^i-\int_t^T\bar z_s^0dW_t^0,~~~~t\in[T-\delta, T]
\]
We could solve
\[
\bar y_t=-M, ~~~
\bar z_t^i=0,~~~\bar z^0_t=0, ~~~t\in[T-\delta, T].
\]

(ii) If we solve \eqref{adjiont} on the interval $[T-k\delta, T-(k-1)\delta](k=1,2,3\cdots)$, and the solution $\{(\bar y_t^i, \bar z_t^i, \bar z_t^0); T-k\delta\leq t\leq T-(k-1)\delta\}$ is Malliavin differentiable, then we could solve the \eqref{adjiont} on the next interval $[T-(k+1)\delta, T-k\delta]$,
\[
\bar y_t^i=\mathbb E [\bar y_{T-k\delta}]+\int_t^{T-k\delta}\widetilde A_s \mathbb E^{\mathcal F^i_s}[\bar y_{s+\delta}]ds,
\]
and
\[
\bar z_t^i=0,~~~\bar z_t^0=0,~~~~t\in[T-(k+1)\delta, T-k\delta].
\]
The optimal control is
\[
\bar u^i_t=-(N_t+\widetilde N_{t+\delta})^{-1}\widetilde B_{t+\delta}\mathbb E^{\mathcal F^i_t}[\bar y_{t+\delta}].
\]
So the $\epsilon$-Nash equilibrium is $(\bar{u}_t^1,\bar{u}_t^2,\cdots,\bar{u}_t^N)$.

Next, we consider a special case that the  coefficients are all constants: $\widetilde A_t=\widetilde A$, $\widetilde B_t=\widetilde B$, $M=1$, $N_t=N$, $\widetilde N_t=\widetilde N$, then the solution of \eqref{adjiont} as follows,
\[
\begin{aligned}
\bar y_{t+\delta}^i&=0,~~\bar z_t^i=0, ~~~\bar z_t^0=0,~~~t\in[T-\delta, T];\\
\bar y_{t+\delta}^i&=-1, ~~\bar z_t^i=0, ~~~\bar z_t^0=0,~~~t\in[T-2\delta, T-\delta];\\
\bar y_{t+\delta}^i&=-1-\widetilde A(T-\delta-t), ~~\bar z_t^i=0, ~~~\bar z_t^0=0,~~~t\in[T-3\delta, T-2\delta];\\
\bar y_{t+\delta}^i&=-1-\widetilde A\delta-\widetilde A(T-2\delta-t)[1+\frac{1}{2}\widetilde A(T-2\delta-t)], ~~\bar z_t^i=0, ~~~\bar z_t^0=0,~~~t\in[T-4\delta, T-3\delta];\\
\cdots\cdots
\end{aligned}
\]

Then, $\epsilon$-Nash equilibrium is $(\bar{u}_t^1,\bar{u}_t^2,\cdots,\bar{u}_t^N)$, where
 $$\bar u^i_t=-\frac{\widetilde B}{N+\widetilde N}\bar y^i_{t+\delta}.$$


\end{document}